\documentclass[10pt]{article}

\usepackage{amsmath} \usepackage{amsfonts}
\usepackage{mathrsfs}  \usepackage{amsthm}
\usepackage[normalem]{ulem}   \usepackage{enumerate}
\usepackage{graphicx}  \usepackage{epstopdf}
\usepackage{subfigure}  \usepackage{lineno}
\usepackage{multirow,bm}
\usepackage{amsmath}
\usepackage{algorithm}
\usepackage{algorithmic}
\usepackage{booktabs}
\usepackage{datetime}
\usepackage{color}
\usepackage{indentfirst}

\newtheorem{proposition} {Proposition}[section]
\newtheorem{assumption} {Assumption}[section]
\newtheorem{theorem}{Theorem}[section]
\newtheorem{lemma}{Lemma}[section]

\theoremstyle{definition}
\newtheorem{definition}{Definition}[section]
\newtheorem{remark}{Remark}[section]
 \newtheorem{example}{Example}[section]

\newcommand{\tto}{\rightarrow} \newcommand{\nn}{\nonumber}
\def \[{\begin{equation}} \def \]{\end{equation}}

\begin{document}

\title{ An inexact augmented Lagrangian method for nonsmooth optimization on Riemannian manifold \thanks{This research was supported by the Natural Science Foundation of China (Grant. 11571074) }}

\author{ Kang-Kang Deng $^\dag$, ~~Zheng Peng \footnote{College of Mathematics and Computer Science, Fuzhou University,  Fuzhou  350108,   China. } \footnote{School of Mathematics and Computational Science, Xiangtan Univeristy, Xiangtan  411105, China. Corresponding author, E-mail: pzheng@fzu.edu.cn. }}
\date{}
\maketitle
 
\noindent {\textbf{Abstract:}  We consider a  nonsmooth optimization problem on Riemannian manifold, whose objective function is  the sum of a differentiable component and a nonsmooth convex   function.   The problem  is  reformulated to a separable form.  We propose a manifold inexact augmented Lagrangian method (MIALM)  for the   considered problem.   By utilizing the Moreau envelope, we get a smoothing  subproblem  at each iteration of the proposed method. Theoretically,  the convergence to critical point of the proposed method is established under suitable assumptions. In particular, if an approximate global minimizer of the iteration subproblem is  obtained at each iteration, we prove that the sequence generated by the proposed method converges to a global minimizer of the origin problem.  Numerical experiments show that, the MIALM is a competitive method compared to some existing methods.



\numberwithin{equation}{section}

\section{Introduction}

Riemannian manifold  optimization is  a  class of constrained  optimization problems, in which the constraint set is a subset of  Riemannian manifold $\mathcal{M}$. It has recently aroused considerable research interests due to the wide  applications in   different fields such as computer vision, signal processing, etc \cite{AbsMahSep2008}. In these applications, manifold $\mathcal{M}$ could be  Stiefel manifold, Grassmann manifold, or symmetric positive definite manifold. Analogy to classical optimization methods in Euclidean space, some Riemannian optimization methods  have been explored,  e.g.,  gradient-type methods\cite{AbsMahSep2008,Boumal2016Global,Zhang2016First}, Newton-type methods\cite{HuaGalAbs2015,YHAG2017,Bortoloti2018Damped} and trust region methods\cite{Absil2007Trust,Boumal2015Riemannian,Huang2015A}.

In this paper, we consider a nonsmooth nonconvex Riemannian manifold optimization problem  as follows
 \begin{equation}\label{problem}\left\{
  \begin{split}
     \min_{X\in\mathbb{R}^{n\times r}} ~& F(X): =  f(X) + g(AX)\\
     \mbox{s.t. }~ & X\in\mathcal{M},
  \end{split}\right.
\end{equation}
where $f: \mathcal{M}\rightarrow \mathbb{R}$ is a smooth but possibly nonconvex function, $g$ is convex  but nonsmooth, and  $\mathcal{M}$ is a Riemannian submanifold embedded in  Euclidean space $\mathbb{E}$.  Many convex or non-convex problems in machine learning applications have the form of problem \eqref{problem},  e.g.,  sparse principle component analysis \cite{zou2006sparse}, sparse canonical correlation analysis \cite{witten2009penalized}, robust low-rank matrix completion
 \cite{cambier2016robust,huang2013optimization} and multi-antenna channel communications \cite{zheng2002communication,gohary2009noncoherent}, etc.

 Absil and Hosseini \cite{absil2017collection} presented  many examples of manifold optimization with nonsmooth objective. We list three representative examples in the following.
\begin{example}[Sparse principle component analysis (SPCA)]
  \begin{equation}
   \left\{ \begin{array}{rl}
      \min\limits_{X\in\mathbb{R}^{n\times r}} & -X^TA^TAX + \lambda\|X\|_1, \\
      \mbox{s.t.} ~& X^TX = I_r.
    \end{array}\right.
  \end{equation}
\end{example}
\begin{example}[Compressed modes in physics (CMs)]
\begin{equation}\left\{
  \begin{split}
     \min_{\Psi\in\mathbb{R}^{n\times r}} & tr(\Psi^T\Delta \Psi) + \mu \|\Psi\|_1, \\
     \mbox{s.t. } ~~ & \Psi^T\Psi = I_r.
  \end{split}\right.
\end{equation}
\end{example}
\begin{example}[Robust low-rank matrix completion]
  \begin{equation}\left\{
    \begin{array}{rl}
      \min\limits_{X\in\mathbb{R}^{n\times n}} & \|P_{\Omega}(X-M)\|_1, \\
      \mbox{s.t.} & X\in\mathcal{M}_r: = \{X | \mbox{ rank }(X) = r\}.
    \end{array}\right.
  \end{equation}
\end{example}

Problem \eqref{problem} is reformulated to a separable form in this paper, and then  a manifold inexact augmented Lagrangian method (MIALM) is proposed for the resulting  separable optimization problem. The iteration subproblem of the MIALM is formulated to a smooth optimization problem by utilizing the Moreau envelope, it could be solved by some classical Riemannian optimization methods such as  Riemannian gradient/Newton/Quasi-Newton method. This algorithmic framework is adapted from \cite{li2018highly, lin2019efficient, dhingra2018proximal} for  classical nonsmooth composite problem in Euclidean space, which has drawn significant research attentions.  
The convergence to critical point of the proposed  MIALM method is established under some mild assumptions. In particular, under the assumption that  an approximate global minimizer of the iteration subproblem could be obtained, the convergence to   global minimizer of the original problem is proved. Numerical experiments show that, the MIALM is competitive  compared to some existing methods.

The rest of this paper is organized as follows.  Some related works on nonsmooth manifold optimization problem are summarized in Section \ref{rwks}, and  some preliminaries on  manifold are given in Section \ref{preli}. In Section \ref{paml}, a manifold inexact augmented Lagrangian method is proposed and the iteration subproblem solver  is presented.   The convergence   of  the proposed method is established in Section \ref{convergence}. Numerical results on  compressed modes problems in physics and sparse PCA are reported in Section \ref{sec-ex}. Finally,  Section \ref{sec-con}   concludes this paper by some remarks.  

\section{Related works}\label{rwks}
We summarize some related works for  nonsmooth optimization problem on manifold in this section.  The existing results mainly focused on  two classes of  nonsmooth manifold optimization problem: nonsmooth optimization problem with locally Lipschitz objective function, and structured optimization problem having the form of problem \eqref{problem}.  

Grohs and Hosseini \cite{Grohs2015} proposed the $\epsilon$-subgradient algorithm for minimizing a locally Lipschitz function on Riemannian manifold. By utilizing $\epsilon$-subgradient-oriented descent directions and the generalized Wolfe line-search  on Riemannian manifold, Hosseini, Huang and Yousefpour  \cite{doi:10.1137/16M1108145} presented  a nonsmooth Riemannian line search algorithm and  established the convergence to a stationary point.  Grohs \cite{Grohs2016Nonsmooth} presented  a nonsmooth trust region algorithm for minimizing locally Lipschitz  objective function on Riemannian manifold.  The iteration complexity of these subgradient algorithms was also investigated in \cite{Bento2017Iteration} and \cite{ferreira2019iteration}. In \cite{doi:10.1137/16M1069298} and \cite{burke2018gradient}, the authors proposed the Riemannian gradient sampling algorithms. At each iteration of  these  Riemannian gradient sampling methods,  the subdifferential of the objective function  is approximated by the convex hull of transported  gradients of  nearby points, and the nearby points are randomly generated in the tangent space of the current iterate. 

Some proximal point algorithms on Riemannian manifold were   investigated in the recent.  Bento, Ferreira and Melo \cite{Bento2017Iteration} analyzed the iteration complexity of a proximal point algorithm on Hadamard manifold  having non-positive sectional curvature.  Bento, et al \cite{deCarvalhoBento2016} gave the full convergence for any bounded sequence generated by the proximal point method,   without assumption on  the sign of the sectional curvature on manifold.   The  Kurdyka-\L ojasiewicz inequality on Riemannian manifold is a powerful tool for convergence analysis of optimization methods on manifold.  Bento, et al \cite{bento2011convergence} analyzed the full convergence of a steepest descent method and  a proximal point method via Kurdyka-\L ojasiewicz inequality. Seyedehsomayeh \cite{hosseini2015convergence} proposed a subgradient-oriented descent method and proved that,  if the objective function has the Kurdyka-\L ojasiewicz property,   the  sequence generated by the subgradient-oriented descent method  converges to a singular critical point.

By  a separable reformulation of problem \eqref{problem}, the variable involving Riemannian manifold constraint and that one involving nonsmooth term could be handled separately. To do so,  it  results in two tractable subproblems. Based on this idea,  Lai, et al \cite{Lai2014A} proposed a splitting of orthogonality constraints (SOC) method  for a special case of problem \eqref{problem}, in which  $f\equiv 0$ and $A = I$, $\mathcal{M}$ is a Stiefel manifold. That is 
\begin{equation}\label{socsub}\left\{
\begin{split}
  \min_X ~& g(X), \\
  \mbox{s.t.} ~& X\in\mathcal{M}.
\end{split}\right.
\end{equation}
To solve problem \eqref{socsub}, the SOC method considered the following separable reformulation:
\begin{equation}\left\{
\begin{split}
  \min_{X,Y}~ & g(Y), \\
  \mbox{s.t.} ~& X\in\mathcal{M}, ~X = Y.
\end{split}\right.
\end{equation}
The associated partial augmented Lagrangian function is
\begin{equation}
\mathcal{L}_{\beta} := g(Y) - \left<\Lambda, X-Y\right> + \frac{\beta}{2}\|X-Y\|_F^2
\end{equation}
where $\Lambda$ is the Lagrangian multiplier, and $\beta$ is a penalty parameter. The SOC method updates iterate via
\[\label{socit}
  \left\{
  \begin{array}{lc}
    X^{k+1} = \arg\min\limits_{X\in\mathcal{M}} \frac{\beta}{2} \|X - Y^k - \frac{1}{\beta} \Lambda^k\|_F^2, \\
    Y^{k+1} =\arg\min g(Y) + \frac{\beta}{2} \|X^{k+1} - Y - \frac{1}{\beta} \Lambda^k\|_F^2, \\
    \Lambda^{k+1} = \Lambda^k - \beta (X^{k+1} - Y^{k+1}).
  \end{array}
  \right.
\]
The  $X$-subproblem  is   ``easy"   via projection on $\mathcal{M}$,  and  the   $Y$ -subproblem  is often structured  in real applications.

Chen, et al \cite{chen2016augmented} proposed a proximal alternating minimization augmented Lagrangian (PAMAL) method of multipliers  for  problem \eqref{problem} with $A = I$ and $\mathcal{M} = St_n$. Specifically, the PAMAL method first reformulates the problem to:
\begin{equation}\label{pamal}\left\{
\begin{split}
  \min\limits_{X,Y,Q}~ & f(Y) + h(Q), \\
  \mbox{s.t.}~ & X = Y,  X = Q,  X\in\mathcal{M}.
  \end{split}\right.
\end{equation}
Then it considers the augmented Lagrangian method of multipliers framework  aiming to obtain the  solution for the  jointed  variable $(X,Y,Q)$ at each iteration. The iterate is produced  by
\begin{equation}\label{pamalit}
\left\{
  \begin{split}
     (X^{k+1},Y^{k+1},Q^{k+1}) &= \arg\min_{X,Y,Q}\mathcal{L}_{\beta}(X,Y,Q;\Lambda_1^k,\Lambda_2^k), \\
     \Lambda_1^{k+1} & = \Lambda_1^{k} - \beta(X^{k+1} - Y^{k+1}), \\
     \Lambda_2^{k+1} & = \Lambda_2^k - \beta(X^{k+1} - Q^{k+1}),
  \end{split}
  \right.
\end{equation}
where $\mathcal{L}_{\beta}$ is the augmented Lagrangian function associated to \eqref{pamal}. The subproblem on the jointed variable $(X, Y, Q)$  is intractable, hence the authors proposed a proximal alternating minimization method to handle it.    Hong, et al \cite{Hong2017Nonconvex} considered a  more general form where $\mathcal{M}$ is the  generalized orthogonal constraint, and proposed a PAMAL-type algorithm in which a proximal alternating linearized minimization method was used for iteration subproblem. 

Kovnatsky, et al \cite{Kovna2015} proposed a manifold ADMM (MADMM)  for a general manifold  optimization problem as follows
\[\label{madmm}\left\{
\begin{split}
  \min_{X,Y} ~& f(X) + g(Y) \\
  \mbox{s.t.} ~& AX = Y, X\in\mathcal{M}.
  \end{split}\right.
\]
The associated partial augmented Lagrangian function is   $$\mathcal{L}_{\beta}(X, Y;\Lambda) : = f(X) + g(Y) - \left<\Lambda,AX-Y\right> + \frac{\beta}{2} \|AX-Y\|_F^2.$$The MADMM has the iterate as follows
 \begin{equation}
   \left\{
   \begin{split}
      X^{k+1} = &\arg\min_{X\in\mathcal{M}}\mathcal{L}_{\beta}(X,Y^{k},\Lambda^k) \\
      Y^{k+1} = &\arg\min_{Y}\mathcal{L}_{\beta}(X^{k+1},Y,\Lambda^k) \\
      \Lambda^{k+1} =& \Lambda^k - \beta(AX^{k+1} - Y^{k+1})
   \end{split}
   \right.
 \end{equation}
More recently, Chen, et al \cite{chen2018proximal} proposed a manifold proximal gradient method (ManPG) for problem \eqref{problem} with $A = I$, i.e.
\begin{equation}
  \min_X f(X) + g(X), \ \ \mbox{s.t. } X\in\mathcal{M}
\end{equation}
At the $k$-th iteration, the search direction $D^k$ of ManPG is obtained by
\begin{equation}\label{D-sub}\left\{
\begin{split}
  \min_D ~&  \left<D,\mbox{grad}f(X^k)\right> + \frac{\beta}{2}\|D\|_F^2 + g(X^k+D), \\
  \mbox{s.t.} ~& D\in T_{X^k}\mathcal{M},
  \end{split}\right.
\end{equation}
where $D\in T_{X^k}\mathcal{M}$ can be represented by a linear system $\mathcal{A}_k(D) = 0$ . The subproblem \eqref{D-sub} is solved by applying the semi-smooth Newton method to the KKT system.  The next iterate $X^{k+1}$ is then obtained  by $$X^{k+1} = R_{X^k}(\alpha_k D^k).$$

\section{Preliminaries} \label{preli}

\subsection{Riemannian manifold optimization} 

Let $\mathcal{M}$ be  a smooth manifold, and $\mathbb{E}$ be the Euclidean space. The tangent space of $\mathcal{M}$ at $x\in\mathcal{M}$ is denoted by $T_x\mathcal{M}$. A Riemannian manifold $(\mathcal{M},\left<\cdot,\cdot\right>)$ is a smooth manifold equipped with  inner product $\left<\cdot,\cdot\right>_x$ on each point $x\in \mathcal{M}$. Let $(U,\varphi)$ be a chart, where $U$ is an open set with $x\in U \subset\mathcal{M} $ and $\varphi$ is a homeomorphism between $U$ and  open set $\varphi(U)\subset \mathbb{E}$. Given a  Riemannian manifold $\mathcal{M}$, the chart  exists at each point $x\in\mathcal{M}$.
\begin{definition}[Riemannian Gradient]
Riemannian gradient, denoted by  $\mbox{grad} f(x) \in T_x\mathcal{M}$,  is  the unique tangent vector satisfying
\begin{equation}\label{grad}
  \left< \mbox{grad}f(x), \xi \right>_x = df(x)[\xi], ~\forall ~\xi\in T_x\mathcal{M}.
\end{equation}
\end{definition}
If $\mathcal{M}$ is an embedded manifold of  $\mathbb{E}$, the Riemannian metric between $u,v\in T_x\mathcal{M}$ could be introduced by  an inner product in $\mathbb{E}$, i.e. $\left<u, v\right>_x = \left<u, v\right>$, where the later is  classical  inner product in $\mathbb{E}$. In the sense, we have
\begin{equation}
  \mbox{grad}f(x) = \mbox{Proj}_{T_x\mathcal{M}}(\nabla f(x))
\end{equation}
where $\nabla f(x)$ is  the gradient in $\mathbb{E}$, $\mbox{Proj}_{T_x \mathcal{M}}$ is a  projection   on   tangent space $T_x \mathcal{M}$.
\begin{definition}[Riemannian Hessian]
Given a smooth function $f:\mathcal{M}\rightarrow \mathbb{R}$, the Riemannian Hessian of $f$ at  $x$ in $\mathcal{M}$ is  linear mapping $\mbox{Hess}f(x)$ of $T_x\mathcal{M}$ into itself,  defined by
\begin{equation}
  \mbox{Hess}f(x)[\xi_x]  = \nabla_{\xi_x}\mbox{grad}f(x)
\end{equation}
for   $\forall ~\xi_x \in T_x\mathcal{M}$, where $\nabla$ is the Riemannian connection on $\mathcal{M}$.
\end{definition}

\begin{definition}[Retraction]\label{retr}
  A retraction on  manifold $\mathcal{M}$ is a smooth mapping $R: T\mathcal{M}\rightarrow \mathcal{M}$ which has the following properties:  let $R_x: T_x\mathcal{M} \rightarrow \mathcal{M}$  be the restriction of  $R$ to $T_x\mathcal{M}$, then
\begin{itemize}
  \item $R_x(0_x) = x$, where $0_x$ is zero element of $T_x\mathcal{M}$
  \item $dR_x(0_x) = id_{T_x\mathcal{M}}$,  where $id_{T_x\mathcal{M}}$ is the identity mapping on $T_x\mathcal{M}$
\end{itemize}
\end{definition}

\begin{definition}[Vector Transport]
  The vector transport $\mathcal{T}$ is a smooth mapping with
  \begin{equation}
    T\mathcal{M}\oplus T\mathcal{M} \rightarrow T\mathcal{M}:(\eta_x,\xi_x)\mapsto \mathcal{T}_{\eta_x}(\xi_x)\in T\mathcal{M},  \forall ~x\in\mathcal{M},
  \end{equation}
   where  $\mathcal{T}$ satisfies that
  \begin{itemize}
    \item $\mathcal{T}_{0_x}\xi_x = \xi_x$ holds for   $\forall ~\xi_x\in T_x\mathcal{M}$;
    \item $\mathcal{T}_{\eta_x}(a\xi_x + b\zeta_x)  = a \mathcal{T}_{\eta_x}(\xi_x) + b\mathcal{T}_{\eta_x}(\zeta_x)$.
  \end{itemize}
\end{definition}

\begin{definition}[The Clarke subdifferential on Riemannian manifold]
  For a locally Lipschitz continuous function $f$ on $\mathcal{M}$, the Riemannian generalized directional derivative of $f$ at $x\in\mathcal{M}$ in direction $v\in T_x\mathcal{M}$  is given by 
    \begin{equation}
    f^{\circ}(x;v) = \lim_{y\rightarrow x}\sup_{t\downarrow 0} \frac{f\circ \varphi^{-1} (\varphi(y) + t \mbox{D}\varphi(y)[v]) - f\circ \varphi^{-1} (\varphi(y))}{t},
  \end{equation}
  where $(\varphi,U)$ is coordinate chart at $x$. The generalized gradient or the Clarke subdifferential of $f$ at $x\in\mathcal{M}$  is
  \begin{equation}
    \partial f(x) = \{\xi\in  T_x\mathcal{M} : \left<\xi,v\right>_x \leq f^{\circ} (x;v), \forall v\in T_x \mathcal{M}\}.
  \end{equation}
\end{definition}
Consider a Riemannian manifold minimization problem
\begin{equation}\label{constraint problem}\left\{
  \begin{split}
     \min_x  ~~& f(x) \\
     \mbox{s.t. }~~ & c_i(x) = 0 ,  i=1,\cdots,m,\\
                    & x\in\mathcal{M}.
  \end{split}\right.
\end{equation}
Let $\Omega := \{x\in\mathcal{M}:c_i(x) = 0,i=1\cdots,m\}$. Given $x^* \in \Omega$, assume that the Linear Independent Constraint Qualification (LICQ) holds at $x^*$, then   normal cone $\mathcal{N}_{\Omega}(x^*)$ is   \cite{Yang2014}:
\begin{equation}
  \mathcal{N}_{\Omega}(x^*) = \left\{\left.\sum_{i=1}^{m}\lambda_i\mbox{grad}c_i(x^*) \right| \lambda\in\mathbb{R}^m\right\}
\end{equation}
For the first-oder optimality condition of problem \eqref{constraint problem}, we have
\begin{lemma}[\cite{Zhang2019},Proposition 2.7]\label{optimality}
 If  $x^*\in\Omega$, and 
   \begin{equation}
    \partial f(x^*) \cap (-\mathcal{N}_{\Omega}(x^*)) \neq \emptyset,
  \end{equation}
  then $x^*$ is a stationary solution of  problem \eqref{constraint problem}.
\end{lemma}

\subsection{Proximal operator and retraction-smooth}
For a proper, convex and low semicontinuous function $g:\mathbb{E}\rightarrow\mathbb{R}$, the proximal operator with  parameter $\mu\geq 0$, denoted by $\mbox{prox}_{\mu g}$, is defined by
\begin{equation}\label{proximal}
  \mbox{prox}_{\mu g} (v)  :=\arg\min_x \{g(x) + \frac{1}{2\mu} \|x-v\|^2\}.
\end{equation}
 The associated Moreau envelope is   a function $M: \mathbb{E}\rightarrow\mathbb{R}$  given by
\begin{equation}\label{More}
\begin{split}
  M_{\mu g} (v) :& = \min_x \{g(x) + \frac{1}{2\mu} \|x-v\|^2\} \\
                 & = g(\mbox{prox}_{\mu g} (v)) + \frac{1}{2\mu} \|\mbox{prox}_{\mu g} (v)-v\|^2.
  \end{split}
\end{equation}
The Moreau envelope is a continuously differentiable function, even when $g$ is not.  
\begin{lemma}[Theorem 6.60 in \cite{beck2017first}]\label{lem-morea}
  Let $g:\mathbb{E}\rightarrow\mathbb{R}$ be a proper closed and convex function, and $\mu\geq 0$. Then $M_{\mu g}$ is $\frac{1}{\mu}$-smooth in $\mathbb{E}$, and   for $\forall ~v\in\mathbb{E}$ one has
   \begin{equation}\label{Morediff}
  \nabla M_{\mu g}(v) = \frac{1}{\mu}(v - \mbox{prox}_{\mu g} (v)).
\end{equation}
\end{lemma}
Lemma \ref{lem-morea} states that,  the Moreau envelope is   continuously differentiable   in Euclidean space $\mathbb{E}$.  Next we will show the relationship between Retraction smoothness in submanifold of Euclidean space and smoothness in Euclidean space.
\begin{definition}[Retraction-Smooth]
  A function $f:\mathcal{M}\rightarrow\mathbb{R}$ is said to be retraction $\ell$-smooth if,  for   $\forall~x, y\in\mathcal{M}$ it holds that
  \begin{equation}
    f(y) \leq f(x) + \left<\mbox{grad}f(x), \xi\right>_x + \frac{\ell}{2}\|\xi\|_x^2,
  \end{equation}
  where $\xi\in T_x\mathcal{M}$ and $R_x(\xi) = y$.
\end{definition}
Let $\mathcal{M}$ be a Riemannian submanifold of $\mathbb{E}$. The following lemma states that, if $f:\mathbb{R}^n \rightarrow \mathbb{R}$ has Lipschitz continuous gradient, then $f$ is also retraction smooth on $\mathcal{M}$.
\begin{lemma}\label{retr-sm}[Lemma 4 in \cite{Boumal2016Global}]\label{retract-euclidean}
  Let $\mathbb{E}$ be a Euclidean space (for example, $\mathbb{E} = \mathbb{R}^n$) and  $\mathcal{M}$ be a compact Riemannian submanifold of $\mathbb{E}$. 
  If $f:\mathbb{E}\rightarrow \mathbb{R}$ has Lipschitz continuous gradient in the convex hull of $\mathcal{M}$, then  there  exists a positive constant $\ell_g$  such that 
  \begin{equation}\label{lemma4}
    f(R_{x_k}(\eta)) \leq f(x_k) + \left<\eta,\mbox{grad}f(x_k)\right> + \frac{\ell_g}{2}\|\eta\|^2
  \end{equation}
  holds at $\forall~\eta\in T_{x_k}\mathcal{M}$.
\end{lemma} Lemma \ref{retr-sm} was proved in \cite{Boumal2016Global}. For the sake of completeness, we give  a proof as follows.
\begin{proof}
  By Lipschitz continuity, $\nabla f$ is Lipschitz  along any line segment in $\mathbb{E}$ jointing $x$ and $y$. Hence, there exists $\ell>0$ such that
\[\label{lips}
    f(y) \leq f(x) + \left<\nabla f(x), y-x\right> + \frac{\ell}{2}\|y-x\|^2,~ \forall x, y\in\mathcal{M}.
 \]
It also holds if $y = R_x(\eta), ~\forall \eta\in T_x\mathcal{M}$. 
Since $\mbox{grad}f(x)$ is a orthogonal projection of $\nabla f(x)$ on $T_x\mathcal{M}$, we have
  \begin{equation}\label{gradf}
  \begin{split}
    \left<\nabla f(x), R_x(\eta) - x\right> & = \left<\nabla f(x), \eta+ R_x(\eta) - x -\eta\right> \\
                                                      & = \left<\mbox{grad} f(x), \eta\right> + \left<\nabla f(x),  R_x(\eta) - x -\eta\right>.
      \end{split}
  \end{equation}
It is easy to deduce from \eqref{lips} and \eqref{gradf} that
  \begin{equation}
    f(R_x(\eta)) \leq f(x) + \left<\mbox{grad} f(x), \eta\right>  + \frac{\ell}{2} \|R_x(\eta) - x\|^2 + \|\nabla f(x)\|\|R_x(\eta)-x-\eta\|. \nn
  \end{equation}
  Since $\nabla f(x)$ is continuous on  compact set $\mathcal{M}$, there exists $G>0$  such that $\|\nabla f(x)\|\leq G, ~\forall~x\in\mathcal{M}$.    By Definition \ref{retr} and the compactness of manifold, there exists   $\alpha,\beta \geq 0$ such that, for all $x\in\mathcal{M}$ and all $\eta\in T_x\mathcal{M}$, we have
  \begin{equation}
     \|R_x(\eta) - x\| \leq \alpha\|\eta\|^2, \mbox{and}~
       \|R_x(\eta) - x -\eta\| \leq \beta\|\eta\|^2. \nn
  \end{equation}
Hence
    \begin{equation}
    f(R_x(\eta)) \leq f(x) + \left<\mbox{grad} f(x), \eta\right>  + \left(\frac{\ell}{2}\alpha^2 + G\beta \right)\|\eta\|^2. \nn
  \end{equation}
  Let $\ell_g = \left(\frac{\ell}{2}\alpha^2 + G\beta \right)$, we have \eqref{lemma4} and complete the proof.
  \end{proof}

\section{The proposed method} \label{paml}
 
\subsection{Problem reformulation}

For regularity, we need the following assumptions on problem \eqref{problem}.
\begin{assumption}\label{assum}
\begin{itemize}
\item[]
\item[A.] $\mathcal{M}$ is a compact Riemannian submanifold embedded in Euclidean space $\mathbb{E}$;
\item[B.] $f$ is smooth but  not necessary convex,  
 $g$ is  a nonsmooth  convex function on $\mathbb{E}$, $A\in \mathbb{R}^{d\times n}$ and $\partial g(Y)$ is uniformly bounded for   $ \forall Y\in \mathbb{R}^{d\times r}$, where $\partial g(Y)$ is   subdifferential of $g$ at $Y$.
\end{itemize}
\end{assumption}

By introducing auxiliary variable $Y = AX $, problem \eqref{problem} can be reformulated to
\begin{equation}\label{problem2}\left\{
\begin{split}\min_{X, Y} ~~f(X) + g(Y) \\
  \mbox{s.t. } AX = Y, ~X\in\mathcal{M}.
  \end{split}\right.
\end{equation}
The partial Lagrangian function associated to problem \eqref{problem2} is
\begin{equation}\label{Lagrangian}
  L(X,Y; Z): = f(X) + g(Y) - \left<Z, AX-Y \right>
\end{equation}
By  Lemma \ref{optimality},   the KKT system  of  problem \eqref{problem2} is  as follows:
\begin{proposition}
  Suppose in  problem \eqref{problem2} that  $f$ is  smooth with Lipschitz continuous gradient and $g$ is convex and locally Lipschitz continuous. Then,  $(X^*, Y^*)$  satisfies the KKT conditions if there exists a Lagrange multiplier $Z^*$ such that
\begin{equation}\label{KKT}
    \left\{
    \begin{array}{rl}
      0 \in& \mbox{Proj}_{T_{X^*}\mathcal{M}} (\nabla f(X^*) -A^TZ^*) , \\
      0 \in &\partial g(Y^*) + Z^* ,\\
      &AX^* = Y^*.
    \end{array}
    \right.
  \end{equation}
\end{proposition}

\subsection{Manifold inexact augmented Lagrangian method}

The augmented Lagrangian associated with \eqref{problem2} is
\begin{equation}\label{augmented Lagrangian}
\begin{split}
  \mathcal{L}_{\rho}(X,Y; Z) & = L(X,Y; Z) + \frac{\rho}{2} \|AX-Y\|_F^2 \\
                           & = f(X) + g(Y) - \left<Z, AX-Y \right> + \frac{\rho}{2} \|AX-Y\|_F^2.
  \end{split}
\end{equation}

For a given $(X^k, Y^k, Z^k)$,  the next iterate generated by our manifold inexact augmented Lagrangian method (MIALM) is given by
\begin{equation}\label{ALM-step}
  \left\{
  \begin{array}{rl}
   (X^{k+1},Y^{k+1}) &= \arg\min\limits_{X\in\mathcal{M},Y}
    \mathcal{L}_{\rho}(X,Y; Z^k), \\
    Z^{k+1} & = Z^k - \rho (AX^{k+1} - Y^{k+1}).
  \end{array}
  \right.
\end{equation}
The $(X, Y)$- subproblem is intractable due to the nonsmoothess and joint minimization. Hence,  an efficient Riemannian optimization method  should be proposed for this subproblem in MIALM \eqref{ALM-step}. Notice that, for  fixed $\rho>0$ and $Z$ we aim to solve
\begin{equation}\label{subproblem}
  \min_{X\in\mathcal{M},Y\in\mathbb{R}^{d\times r}} \Psi(X,Y) := \mathcal{L}_{\rho} (X,Y; Z)
\end{equation}
Let
\begin{equation} \label{psi}
\begin{split}
  \psi_{Z}(X): =  &~~\inf_{Y} \Psi(X,Y) \\
  = &~~f(X) + g(\mbox{Prox}_{ g/\rho}(AX-\mu Z)) \\
             & + \frac{\rho}{2}\|AX-\frac{1}{\rho} Z - \mbox{Prox}_{\mu g}(AX-\frac{1}{\rho} Z)\|_F^2 - \frac{1}{2\rho}\|Z\|_F^2.
  \end{split}
\end{equation}
 The new iterate  $(\bar{X},\bar{Y})$ could be produced sequentially   by
 \begin{equation}
   \bar{X} = \arg\min_{X\in\mathcal{M}} \psi_{Z}(X), ~~ \bar{Y} = \mbox{Prox}_{ g/\rho}(A\bar{X}-\frac{1}{\rho} Z).
 \end{equation}
In the sense, the MIALM iterate is  rewritten to
\begin{equation}\label{ALM-step1}
  \left\{
  \begin{array}{lc}
    X^{k+1} = \arg\min_{X\in\mathcal{M}} \psi_{Z^k}(X), \\
    Y^{k+1} = \mbox{Prox}_{ g/\rho}(AX^{k+1}-\frac{1}{\rho} Z^k),\\
    Z^{k+1} = Z^k - \rho (AX^{k+1} - Y^{k+1}).
  \end{array}
  \right.
\end{equation}
By \eqref{Morediff}, we have
\begin{equation}
\begin{split}
  \nabla \psi_{Z}(X) & = \nabla f(X) + \rho A^T\left(AX-\frac{1}{\rho} Z - \mbox{Prox}_{\mu g}(AX-\frac{1}{\rho} Z)\right) \\
                           & = \nabla f(X) + \rho A^T\left( \mbox{Prox}_{ \rho g^*}(AX-\frac{1}{\rho} Z)\right)
  \end{split}
\end{equation}
where  $g^*$ is the conjugate operator of  $g$ and defined by $g^*(x) = \sup_v\{\left<x,v\right> - g(v)\}$.  By Lemma \ref{retract-euclidean}, $\psi_Z(\cdot)$   is retraction smooth over Riemannian manifold $\mathcal{M}$, and its Riemannian gradient is   $$\mbox{grad} \psi_{Z}(X) = \mbox{Proj}_{T_X\mathcal{M}}(\nabla
\psi_{Z}(X)).$$  Thus, at the $k$-th iteration, the $X$-subproblem  is identical to find $X^{k+1}$ such that $$\mbox{grad} \psi_{Z^k}(X^{k+1}) = 0.$$

Algorithm \ref{alg2} below summarizes the proposed manifold inexact augmented Lagrangian method in details.
\begin{remark}
  
  \begin{enumerate}
    \item[1)] The proposed  method is an ALM-type method. 
    The complexity of  $X$-subproblem is as same as that of  MADMM. However, our method  obtains a joint optimal solution which guarantees  the convergence, while the MADMM does not.
    \item[2)] All efficient Riemannian optimization methods are applicable for the $X$-subproblem, e.g., Riemannian gradient method, Riemannian Newton method,  etc. 
    \item[3)] The proposed method is utilizable for smooth Riemannian optimization problem under set-constrained, in which $g(X) = \delta_{\Omega}(X)$, the indictor function of constraint set $\Omega$.
  \end{enumerate}
\end{remark}

\begin{algorithm}[tb]
   \caption{Manifold inexact augmented Lagrangian method for problem \eqref{problem}}
   \label{alg2}
\begin{algorithmic}[1]
   \STATE {\bfseries Input:} Let $Z_{\min}< Z_{\max}$,  $X_0\in\mathcal{M}$,  $\bar{Z}_0\in \mathbb{R}^{d\times r}$. Given  $\epsilon_{\min}\ge 0$, $\epsilon_0>0$,  $\rho_0>1$, $\sigma>1, 0<\tau<1$. 
   \FOR{$k=0,1,\cdots$}
   \STATE Produce the next iterate $(X^{k+1}, Y^{k+1})$:   get $X^{k+1}$  by solving problem 
    \begin{equation}\label{xit}
            \min_{X\in\mathcal{M}}\psi_{\bar{Z}^k}(X)
          \end{equation}
 inexactly   with  a tolerance $\epsilon_k$ where  $\{\epsilon_k\}_{k\in\mathbb{N}}\downarrow 0$; let
   \begin{equation}\label{yit}
     Y^{k+1} = \mbox{Prox}_{ g/\rho_k}(AX^{k+1} - \bar{Z}^k).
       \end{equation}

   \STATE Update Lagrangian multiplier $Z^{k+1}$ by
   \begin{equation}\label{zit}
     Z^{k+1} = \bar{Z}^k - \rho_k(AX^{k+1} - Y^{k+1})
   \end{equation}
   \STATE  Project $Z^{k+1}$ onto $\{Z: Z_{\min}\leq Z \leq Z_{\max}\}$ to get $\bar{Z}^{k+1}$.
   \STATE Update  penalty parameter by
   \begin{equation}\label{parait}
     \rho_{k+1} = \left\{
     \begin{array}{cc}
       \rho_{k}, & \mbox{ if } ~\|AX^{k+1} - Y^{k+1}\|_{\infty} \leq \tau \|AX^{k} - Y^{k}\|_{\infty} \\
       \sigma \rho_{k},  & \mbox{otherwise}
     \end{array}
     \right.
   \end{equation}
   \ENDFOR
\end{algorithmic}
\end{algorithm}

\subsection{Riemannian optimization subproblem} \label{sec-sub}
The main computational cost of Algorithm \ref{alg2} is to solve the $X$-subproblem. It is a smooth optimization problem on Riemannian manifold.  
The $X$-subproblem could  be restated  as follows
\begin{equation}\label{x-subproblem}
  \min_X \psi(X),~ ~\mbox{s.t.}~ X\in\mathcal{M}.
\end{equation}
where $\psi=\psi_{\bar{Z}}$   given by \eqref{psi}.  
 It is a retraction smooth function, so  problem \eqref{subproblem} can be solved by some Riemannian gradient methods  including Riemannian gradient descent (RGD), Riemannian conjugate gradient (RCG) and Riemannian trust region (RTR) method,  etc.  In this paper,   we adopt a RGD method to  problem \eqref{x-subproblem}, see Algorithm \ref{alg3} for details.

\begin{algorithm}[tb]
   \caption{Riemannian gradient method for subproblem \eqref{x-subproblem} }
   \label{alg3}
\begin{algorithmic}[1]
   \STATE {\bfseries Given: }   $X^0\in\mathcal{M}$,   tolerance $\epsilon>0$.  Let  $\eta^0 = -\mbox{grad}\psi(X^0)$ .
   \STATE {\bfseries Initialize:  } k = 0.

   \WHILE{$\|\eta^k\| \geq \epsilon$}
   \STATE Pick $\eta^k = - \mbox{grad}\psi(X^k) $ and a step size $\alpha_k$, compute
   \begin{equation}
     X^{k+1} = \mbox{R}_{X^k}(\alpha_k \eta^k).
   \end{equation}
   \ENDWHILE
\end{algorithmic}
\end{algorithm}

\section{Convergence analysis}\label{convergence}
For convenience of notation, we rewrite  problem \eqref{problem2}   to a standard constraint optimization problem on manifold. Let $W= [X; Y]\in\mathbb{R}^{(n+d)\times r}$, 
and  $\mathcal{N} = \mathcal{M} \times \mathbb{R}^{d\times r}$ be a product manifold. Then, problem \eqref{problem2} can be rewritten to
\[\label{newpro}
\min_W ~\theta (W), ~ \mbox{s.t. }~ h(W) = 0, ~W\in\mathcal{N}.
\]
 where $\theta (W) = f(X) + g(Y)$, and $h(W)= [A,-I]W \in\mathbb{R}^{d\times r}$. The partial  augmented Lagrangian function associated to problem \eqref{newpro}  is
\[\label{alf}
\mathcal{L}_{\rho}(W; Z) = \theta (W)+ \sum_{i=1}^{d}\sum_{j=1}^{r}Z_{ij}[h(W)]_{ij} + \frac{\rho}{2}\sum_{i=1}^{d}\sum_{j=1}^{r} [h(W)]_{ij}^2
\]
The  KKT conditions of problem \eqref{newpro} are given by
\begin{equation}\label{kkt}
  0\in \partial \theta (W^*)+  \sum_{i=1}^{d}\sum_{j=1}^{r}Z^*_{ij} \mbox{grad}[h(W^*)]_{ij}, ~~ h(W^*) = 0 , ~~W^*\in\mathcal{N},
\end{equation}
 where $ \partial \theta (W^*) $ is Riemannian subdifferential of $\theta$ at $W^*$. The KKT system \eqref{kkt} is identical to \eqref{KKT} because of that $\mathcal{M}$ is a Riemannian submanifold embedded in Euclidean space.
Inspired by Zhang, Yang and Song \cite{Yang2014},   the constraint qualifications of problem \eqref{newpro} is given by:

\begin{definition}[LICQ]
  Linear independence constraint qualifications (LICQ) are said to hold at $W^*\in\mathcal{N}$ for problem \eqref{newpro} if
  \[
  \Big\{\mbox{grad}[h(W^*)]_{ij}| i=1,\cdots,d;j =1,\cdots,r \Big\} \mbox{ are linearly independent in } T_{W^*}\mathcal{N}. \nn
  \]
\end{definition}


We will analyze the convergence of Algorithm \ref{alg2} in the following two cases:
\begin{enumerate}
  \item[1)]The iterate $(X^{k+1},Y^{k+1}) $ is an $\epsilon_k$-stationary point of  iteration subproblem, i.e.,
  \begin{equation}\label{local}
    \|\mbox{grad}\psi_{\bar{Z}^k}(X^{k+1})\| \leq \epsilon_k.
  \end{equation}
  \item[2)] The iterate $(X^{k+1},Y^{k+1}) $ is an $\epsilon_k$-global minimizer of   iteration subproblem, i.e.,
  \begin{equation}\label{global}
    \mathcal{L}_{\rho_k}(W^{k+1};\bar{Z}^k) \leq \mathcal{L}_{\rho_k}(W;\bar{Z}^k) + \epsilon_k, ~\forall ~ W\in\mathcal{N}.
  \end{equation}
\end{enumerate}

\begin{remark}\label{rm3}
  In the case 1), \eqref{local}  is indeed to find $W^{k+1}$ such that $$\delta^{k} \in \partial \mathcal{L}_{\rho_k}(W^{k+1};\bar{Z}^{k}), ~\|\delta^k\| \leq \epsilon_k. $$
\end{remark}

\begin{theorem}\label{localconv}
  Suppose $\{W^k\}_{k\in\mathbb{N}}$ is a sequence generated by Algorithm \ref{alg2},  Assumption \ref{assum}  and \eqref{local} hold.
 Then, sequence  $\{W^k\}_{k\in\mathbb{N}}$ has at least one cluster point. Furthermore, if  $W^*$ is  a cluster point, and  LICQ  holds at $W^*$, then $W^*$ is a KKT point of  problem \eqref{newpro}.
\end{theorem}

  \begin{proof}  To prove the first part of  Theorem \ref{localconv}, we need to show that  sequence $\{W^k\}_{k\in\mathbb{N}}$ is bounded. 
  By Assumption \ref{assum},  $\mathcal{M}$ is a compact manifold,  hence $\{X^k\}$ is bounded. By  $$Y^{k+1} = \mbox{Prox}_{ g/\rho}(AX^{k+1}-\frac{1}{\rho} \bar{Z}^k),$$
  there exists $\nu^k \in \partial g(Y^{k+1})$ such that
$$
    0 = \nu^k - \rho_k(AX^{k+1} - \frac{1}{\rho} \bar{Z}^k - Y^{k+1}).
$$
  Again by Assumption \ref{assum}, $\partial g(Y^{k+1})$ is bounded, and hence $\nu^k$ is also bounded. It is  obvious that $\bar{Z}^k\in [Z_{min}, Z_{max}]$  is   bounded. Since sequence $\{\rho_k\}_{k\in\mathbb{N}}$ is nondescreasing,  we have $\rho_k \geq \rho_0~ (\forall k\in\mathbb{N})$.  Hence  $\{Y^k\}_{k\in\mathbb{N}}$ is bounded.
  In summary,  sequence $\{W^k\}_{k\in\mathbb{N}}$ is bounded.

  Next,  we will show that $W^*$ is a feasible point of \eqref{newpro}.
  By the updating rule of $W$ in Algorithm \ref{alg2}, we have $W^k\in\mathcal{N}$.

    If  $\{\rho_k\}_{k\in\mathbb{N}}$ is bounded,  by the updating rule of $\rho_k$,     there exists a $k_0\in\mathbb{N}$ such that
    $$\|h(W^k)\|_{\infty} \leq \tau\|h(W^{k-1})\|_{\infty},~~ \forall k\geq k_0,$$
where $\tau\in (0, 1)$.  Hence,   $h(W^*) = 0$.

If $\{\rho_k\}$ is unbounded, by Remark \ref{rm3}  we have 
$$\delta^{k} \in \partial \mathcal{L}_{\rho_k}(W^{k+1};\bar{Z}^{k}), ~\|\delta^k\| \leq \epsilon_k. $$
where $\epsilon^k\downarrow 0$ as $k\rightarrow \infty$. Thus there exists $U^k \in \partial \theta(W^k)$  such that
\[\label{wkdelta}
     U^k + \sum_{i=1}^{d}\sum_{j=1}^{r} \left(\bar{Z}^k_{ij} + \rho_k [h(W^k)]_{ij}\right) \mbox{grad} [h(W^k)]_{ij} = \delta^k.
\]
Dividing both sides of \eqref{wkdelta} by $\rho_k$, we have
      \begin{equation}
       \sum_{i=1}^{d}\sum_{j=1}^{r} \left(\bar{Z}^k_{ij}/\rho_k +  [h(W^k)]_{ij}\right) \mbox{grad} [h(W^k)]_{ij} = (\delta^k  - U^k)/\rho_k
    \end{equation}
    where $\{\bar{Z}^k\}$ is bounded, and $\delta^k \downarrow 0$.   Since $\theta(W) = f(X) + g(Y)$, where $g$ is a convex function on $\mathbb{E}$, and $$\partial \theta(W) = \left(
   \begin{array}{c}
     \mbox{grad}f(X) \\
     \partial g(Y)
   \end{array}\right),$$ where $\partial g(Y)$ is a subdifferential (set) in usual sense. Invoked by Proposition B.24(b) in \cite{bertsekas1997nonlinear}, the set $\bigcup_{k\in\mathcal{K}} \partial g(Y^k)$ is bounded because that $\{Y^k\}_{k\in\mathcal{K}}$ is a bounded set. In addition, $f(X)$ is a retraction smooth function, hence the Riemannian gradient sequence $\{ \mbox{grad}f(X^k)\}_{k\in\mathcal{K}}$ is bounded. Thus, we have that $\bigcup_{k\in\mathcal{K}} \partial \theta(W^k)$ is bounded. This means that $\{U^k\}$ is bounded. Taking limits as  $k\in\mathcal{K}$ going to infinity on both sides of \eqref{wkdelta}, and using  the continuity and differentiability  of $ h$, we have,
    \begin{equation}
       \sum_{i=1}^{d}\sum_{j=1}^{r} \left(   [h(W^*)]_{ij}\right) \mbox{grad} [h(W^*)]_{ij}  = 0
    \end{equation}
    Note that LICQ holds at $W^*$, we conclude that $[h(W^*)]_{ij} = 0$ for all  $i, j$.

    Since $\{U^k\}_{k\in\mathcal{K}}$ is bounded, there exists a subsequence $\mathcal{K}_1\subset \mathcal{K}$ such that $\lim\limits_{k\to \infty, k\in\mathcal{K}_1} U^k = U^*$. Recall that $\lim\limits_{k\to \infty, k\in\mathcal{K}_1}W^k = W^*$. We get
    $$U^* \in \partial \theta (W^*) $$
  by the closedness property of the limiting subdifferential. Together with $Z^{k+1}_{ij} = \bar{Z}^k + \rho_k [h(W^k)]_{ij}$ for all $i,j$, one  can get  from Algorithm \ref{alg2} that,  for all $k\in\mathcal{K}_1$,
  \begin{equation}\label{suboptim}
      U^k + \sum_{i=1}^{d}\sum_{j=1}^{r} \left(Z^{k+1}_{ij}\right) \mbox{grad} [h(W^k)]_{ij} = \delta^k
    \end{equation}
   where $\delta^k$ satisfying $\|\delta^k\| \leq \epsilon^k$, and $U^k \in \partial \theta (W^k)$.

  We claim that $\{Z^k\}$ is bounded. Otherwise, assume  $\{Z^k\}$ is unbounded, we have
  \begin{equation}
      \frac{ U^k}{\|Z^{k+1}\|_{\infty}} + \sum_{i=1}^{d}\sum_{j=1}^{r} \left(\frac{ Z^{k+1}_{ij}}{\|Z^{k+1}\|_{\infty}}\right) \mbox{grad} [h(W^k)]_{ij} = \frac{\delta^k}{\|Z^{k+1}\|_{\infty}}\nonumber
    \end{equation}
    Since $\frac{ Z^{k+1}}{\|Z^{k+1}\|_{\infty}} \in [-1,1]$ is bounded, there exists a subsequence $\mathcal{K}_2 \subset \mathcal{K}_1$ such that $\lim\limits_{k\to \infty, k\in\mathcal{K}_2} \frac{ Z^{k+1}}{\|Z^{k+1}\|_{\infty}} = \bar{Z}$, where $\bar{Z}$ is a nonzero matrix. Taking the limit on  $k\in\mathcal{K}_2$ going to infinity,  we  obtain
    \begin{equation}
       \sum_{i=1}^{d}\sum_{j=1}^{r} \bar{Z}_{ij} \mbox{grad} [h(W^*)]_{ij} = 0,
    \end{equation}
    which contradicts the LICQ condition at $W^*$.

   Since $\{U^k\}$ is bounded and $\{\delta^k\} \downarrow 0$,  there exists a subsequence $\mathcal{K}_3 \subset \mathcal{K}_2$ such that $\lim\limits_{k\to \infty, k\in\mathcal{K}_3} U^k = U^*$ and $\lim\limits_{k\to \infty, k \in\mathcal{K}_3} Z^k = Z^*$.  By the continuity of mapping  $\mbox{grad}~h$, and taking limits on  $k\in\mathcal{K}_3$ going to infinity on both sides of \eqref{suboptim}, we have
     \begin{equation}
      U^* + \sum_{i=1}^{d}\sum_{j=1}^{r} \left(Z^*_{ij}\right) \mbox{grad} [h(W^*)]_{ij} = 0.
    \end{equation}
  \end{proof}

\begin{lemma}
  Suppose that $W\in\mathcal{N}= \mathcal{M}\times R^{d\times r}$, and   $\mathcal{M}$ is a stiefel manifold denoted by $St(n,r)$). Then the LICQ   always holds at $\forall ~W\in \mathcal{N}$.
\end{lemma}
\begin{proof}
Let $e_i\in\mathbb{R}^d$ be  a $m$-dimensional coordinate vector in which the $i$-th entry is 1 and 0 for others, and $\bar{e}_j\in\mathbb{R}^r$ be   a $r$-dimensional coordinate vector. Then
  \begin{equation}
    \nabla [h(W)]_{ij} = \left(
    \begin{array}{c}
      A^Te_i\bar{e}_j^T \\
      -e_i\bar{e}_j^T
    \end{array}
    \right), ~~ i=1,\cdots,d; j=1,\cdots, r. \nonumber
  \end{equation}
A basis of the  normal cone of $St(n,r)$ at $X$, denoted by $N_{X}St(n, r)$,  is given by
\begin{equation}
  \left\{X(\bar{e}_i\bar{e}_j^T + \bar{e}_j\bar{e}_i^T): i=1,\cdots,r, j=1,\cdots,r \right\}.\nonumber
\end{equation}
It is easy to  show that, for   $\forall~W\in\mathcal{N}$,  all the vectors in the set
$$\left\{
\left(
    \begin{array}{c}
      A^Te_i\bar{e}_j^T \\
      -e_i\bar{e}_j^T
    \end{array}
    \right), i=1,\cdots, d; j=1,\cdots, r. \right\} \bigcup \left\{
\left(
\begin{array}{c}
  X(\bar{e}_i\bar{e}_j^T + \bar{e}_j\bar{e}_i^T) \\
  0
\end{array}
\right), i=1,\cdots,r; j=1,\cdots, r.
\right\} $$
are linearly independent.  Hence, if there exists $Z$ such that
\[\label{licq}
\sum_{i=1}^{d}\sum_{j=1}^{r} Z_{ij}\nabla [h(W)]_{ij} \in N_W\mathcal{N},
\]
 we have   $Z = 0$.  Since $\mathcal{N}$ is a submanifold of  Euclidean space, it derives immediately   that  \[\nn
\sum_{i=1}^{d}\sum_{j=1}^{r} Z_{ij}\mbox{grad} [h(W)]_{ij} =0,
\]
holds if and only if $Z = 0$.  Which implies    LICQ  holds at  $W$ and  completes the proof.
\end{proof}

Next, we consider the case that  a $\epsilon_k$-global minimizer of the iteration subproblem could be  obtained at each iteration of Algorithm \ref{alg2}.

\begin{theorem}
  Assume that $\{W^k\}_{k\in\mathbb{N}}$ is a sequence generated by Algorithm \ref{alg2},   Assumption \ref{assum} holds,   and  \eqref{global} is satisfied  at each iteration of Algorithm \ref{alg2}. Let $W^*$ be a limit point of $\{W^k\}_{k\in\mathbb{N}}$.  Then we have
\[\label{constob} \sum_{i=1}^{d}\sum_{j=1}^{r}[h(W^*)]_{ij}^2\leq \sum_{i=1}^{d}\sum_{j=1}^{r}[h(W)]_{ij}^2, ~\forall ~W\in\mathcal{N}.\]
\end{theorem}

\begin{proof}
  Consider  the following two cases:  $\{\rho_k\}$ bounded and $\rho_k \rightarrow \infty$.

 If $\{\rho_k\}$ is bounded, then there exists $k_0$ such that $\rho_k = \rho_{k_0}$ for all $k\geq k_0$. Hence 
  $$ \sum_{i=1}^{d}\sum_{j=1}^{r}[h(W^{k+1})]_{ij}^2\leq  \tau\sum_{i=1}^{d}\sum_{j=1}^{r} [h(W^k)]_{ij}^2, i=1,\cdots,m;j=1,\cdots,r.$$
  Which implies that $h(W^k) \rightarrow 0$ as $k \rightarrow \infty$. We have
$ h(W^*) = 0$,   and  \eqref{constob} holds.

  Then to the case  $\rho_k\tto \infty$.  Since $W^*$ is a limit point of $\{W^k\}$, there exists a subsequence $\mathcal{K}\subset \mathbb{N}$ such that 
  \begin{equation}
  \lim_{k\tto \infty, ~k\in \mathcal{K}}W^k = W^*. \nn
  \end{equation}
  Assume by contradiction there exists $W\in\mathcal{N}$ such that
  $$\sum_{i=1}^{d}\sum_{j=1}^{r} [h(W^*)]_{ij}^2\geq \sum_{i=1}^{d}\sum_{j=1}^{r}[h(W)]_{ij}^2.$$
  By the boundedness of $\{\bar{Z}^k\}$ and  $\rho_k \to \infty$, there exist $c>0$ and $k_0\in\mathbb{N}$ such that,  for all $k\in \mathcal{K}$ and $k\geq k_0$ we have
  $$
  \sum_{i=1}^{d}\sum_{j=1}^{r}( [h(W^{k+1})]_{ij} + \frac{1}{\rho_k} \bar{Z}^{k}_{ij})^2\geq \sum_{i=1}^{d}\sum_{j=1}^{r}([h(W)]_{ij} + \frac{1}{\rho_k} \bar{Z}^{k}_{ij})^2 + c.
  $$
  Therefore 
  \begin{equation}
  \begin{split}
    \theta (W^{k+1}) + \frac{\rho_k}{2}  \sum_{i=1}^{d}\sum_{j=1}^{r}( [h(W^{k+1})]_{ij} + \frac{1}{\rho_k} \bar{Z}^{k}_{ij})^2 & \geq \theta (W) + \frac{\rho_k}{2}  \sum_{i=1}^{d}\sum_{j=1}^{r}([h(W)]_{ij} + \frac{1}{\rho_k} \bar{Z}^{k}_{ij})^2  \\
    & + \frac{\rho_k c}{2} + \theta (W^{k+1}) - \theta (W).
    \end{split}\nn
  \end{equation}
  Since $\lim\limits_{k\to \infty, k\in \mathcal{K}} W^k = W^*$, and $\{\epsilon_k\}$ is bounded, there exists $k_1>k_0$ such that,
  for  all $k\in \mathcal{K}, k\geq k_1$ we have
  \begin{equation}
    \frac{\rho_k c}{2} +\theta (W^{k+1}) - \theta (W) >\epsilon_k. \nn
  \end{equation}
  Therefore,
    \begin{equation}
  \begin{split}
    \theta (W^{k+1}) + \frac{\rho_k}{2}  \sum_{i=1}^{d}\sum_{j=1}^{r}( [h(W^{k+1})]_{ij} + \frac{1}{\rho_k} \bar{Z}^{k}_{ij})^2 & \geq \theta (W) + \frac{\rho_k}{2}  \sum_{i=1}^{d}\sum_{j=1}^{r}([h(W)]_{ij} + \frac{1}{\rho_k} \bar{Z}^{k}_{ij})^2 + \epsilon_k. \nn
    \end{split}
  \end{equation}
  This contradicts \eqref{global}. We have \eqref{constob} and complete the proof.
\end{proof}

\begin{theorem}
   In Algorithm \ref{alg2},  let $\epsilon_{\min} = 0$ and $W^*$ be a limit point of sequence  $\{W^k\}_{k\in\mathbb{N}}$. If  iterate  $W^{k+1}$ is a $\epsilon_k$- global minimizer satisfying \eqref{global}, then  $W^*$  is a global minimizer of  problem \eqref{problem2}. Meanwhile,  $X^*$  is a global minimizer of problem \eqref{problem}.
\end{theorem}

\begin{proof}

 By \eqref{global}, we have
     \begin{equation}
  \begin{split}
    \theta (W^{k+1}) + \frac{\rho_k}{2}  \sum_{i=1}^{d}\sum_{j=1}^{r}( [h(W^{k+1})]_{ij} + \frac{1}{\rho_k} \bar{Z}^{k}_{ij})^2 & \leq \theta (W) + \frac{\rho_k}{2}  \sum_{i=1}^{d}\sum_{j=1}^{r}([h(W)]_{ij} + \frac{1}{\rho_k} \bar{Z}^{k}_{ij})^2 + \epsilon_k
    \end{split} \nn
  \end{equation}
  for all $W\in \mathcal{N}$. Since $h(W) = 0$, we get
       \begin{equation}
  \begin{split}
    \theta (W^{k+1}) + \frac{\rho_k}{2}  \sum_{i=1}^{d}\sum_{j=1}^{r}([h(W^{k+1})]_{ij} +  \frac{1}{\rho_k} \bar{Z}^{k}_{ij})^2 & \leq
    \theta (W) + \frac{\rho_k}{2}  \sum_{i=1}^{d}\sum_{j=1}^{r}( \frac{1}{\rho_k} \bar{Z}^{k}_{ij})^2 + \epsilon_k.
    \end{split}\nn
  \end{equation}
Which  implies that
    \begin{equation}\label{des}
  \begin{split}
    \theta (W^{k+1})  \leq \theta (W) + \frac{\rho_k}{2}  \sum_{i=1}^{d}\sum_{j=1}^{r}(\frac{1}{\rho_k} \bar{Z}^{k}_{ij})^2+ \epsilon_k.
    \end{split}
  \end{equation}

 If $\rho_k\to \infty$, by taking limits on both sides of \eqref{des} as  $k\in \mathcal{K}, k\to \infty$,   and using   $\lim\limits_{k\to \infty, k\in \mathcal{K}}\epsilon_k = 0$, we get
    \begin{equation}
    \theta(W^*)  \leq \theta (W), ~\forall~W\in\mathcal{N}.\nn
  \end{equation}

   In case of that  $\{\rho_k\}$ is bounded, there exists $k_0\in\mathbb{N}$ such that $\rho_k = \rho_{k_0}$ for all $k>k_0$. By \eqref{global} we  have
     \begin{equation}
  \begin{split}
    \theta (W^{k+1}) + \frac{\rho_{k_0}}{2}  \sum_{i=1}^{d}\sum_{j=1}^{r}( [h(W^{k+1})]_{ij} + \frac{1}{\rho_{k_0}} \bar{Z}^{k}_{ij})^2 & \leq \theta (W) + \frac{\rho_{k_0}}{2}  \sum_{i=1}^{d}\sum_{j=1}^{r}([h(W)]_{ij} + \frac{1}{\rho_{k_0}} \bar{Z}^{k}_{ij})^2 + \epsilon_k
    \end{split}\nn
  \end{equation}
  for $W\in \mathcal{N}$. Since $h(W) = 0$, we get
   \begin{equation}\label{desc1}
  \begin{split}
    \theta (W^{k+1}) + \frac{\rho_{k_0}}{2}  \sum_{i=1}^{d}\sum_{j=1}^{r}( [h(W^{k+1})]_{ij} + \frac{1}{\rho_{k_0}} \bar{Z}^{k}_{ij})^2 & \leq \theta(W) + \frac{\rho_{k_0}}{2}  \sum_{i=1}^{d}\sum_{j=1}^{r}( \frac{1}{\rho_{k_0}} \bar{Z}^{k}_{ij})^2 + \epsilon_k
    \end{split}
  \end{equation}
  for all $k\geq k_0$. Let $\mathcal{K}_1 \subset \mathcal{K}$   and 
  \begin{equation}
    \lim_{k\to \infty, k\in \mathcal{K}_1} \bar{Z}^k = Z^*. \nn
  \end{equation}
   Taking limits on both sides of \eqref{desc1} as $k\to \infty, k\in \mathcal{K}_1$, and noting that $h(W^*) = 0$, we get
    \begin{equation}
  \begin{split}
    \theta (W^*) + \frac{\rho_{k_0}}{2}  \sum_{i=1}^{d}\sum_{j=1}^{r}(  \frac{1}{\rho_{k_0}} {Z}^{*}_{ij})^2 & \leq \theta (W) + \frac{\rho_{k_0}}{2}  \sum_{i=1}^{d}\sum_{j=1}^{r}( \frac{1}{\rho_{k_0}} {Z}^{*}_{ij})^2.\nn
    \end{split}
  \end{equation}
Hence
       \begin{equation}
    \theta(W^*)  \leq \theta (W), ~\forall~ W\in \mathcal{N},\nn
  \end{equation}
 and  the proof is completed.
  \end{proof}

\section{Experiments} \label{sec-ex}
Numerical experiments for testing the performance of  the proposed  MIALM method, with compared to  some  existing methods  including  SOC \cite{Lai2014A}, PAMAL \cite{Hong2017Nonconvex}, MADMM \cite{Kovna2015} and ManPG \cite{chen2018proximal}, are presented in the current section. All the  methods   are used  to  solve the compressed modes and sparse PCA problem.  In  the  MIALM   and MADMM, the Riemannian manifold optimization subproblem is handled by  ``Manopt", a Matlab toolbox for optimization on manifolds  \cite{JMLR:v15:boumal14a}. In the SOC, PAMAL and ManPG methods,   the code provided by Chen \cite{chen2018proximal} are used (all codes are available in online).  All experiments are run  on a personal computer with 4.0GHz Intel Core i7  CPU and  16 GB RAM.

\subsection{Compressed modes in Physics}

 In physics, the compressed modes problem (CMs)   seeks spatially localized solutions of the independent-particle Schr\"{o}dinger equation:
\begin{equation}
  \hat{H}\phi(x) = \lambda \phi(x), ~~x\in\Omega,
\end{equation}
where $\hat{H} = -\frac{1}{2}\Delta$  and $\Delta$ is  a Laplacian operator. 
Consider the 1$D$ free-electron (FE) model  with $\hat{H} = -\frac{1}{2}\partial_{x^2}$.  By a proper discretization, the compressed modes problem can be reformulated to
\begin{equation}\label{CMs}\left\{
  \begin{split}
     \min_{X\in\mathbb{R}^{n\times k}} & tr(X^TH X) + \mu \|X\|_1, \\
     \mbox{s.t. } ~~ & X^TX = I_d,
  \end{split}\right.
\end{equation}
where $H$ is the discretized Schr\"{o}dinger operator, $\mu$ is a regularization parameter. The interesting readers are referred to \cite{ozolicnvs2013compressed} for more details.
For   problem \eqref{CMs}, both SOC and PAMAL consider the identical form as follows:
\begin{equation}\left\{
  \begin{split}
     \min_{\Psi,Q,P\in\mathbb{R}^{n\times r}} & tr(X^T H X) + \mu \|Q\|_1, \\
     \mbox{s.t.  }~~ & Q = X, P = X, P^TP = I_r.
  \end{split}\right.
\end{equation}
The MADMM handles the separable reformulation of the form
  \begin{equation}\left\{
  \begin{split}
     \min_{\Psi,Q\in\mathbb{R}^{n\times r}} & tr(X^T H X) + \mu \|Q\|_1, \\
     \mbox{s.t.  }~~ & Q = X,  X^TX = I_r.
  \end{split}\right.
\end{equation}
In our experiments, the domain $\Omega:= [0, 50]$ is discretized with $n$ equally spaced nodes. The parameters of our MIALM are set to : $\tau = 0.99, \sigma = 1.05, \rho_0 = \lambda_{\max}(H)/2, Z_{\min} = -100\cdot 1_{d\times r}, Z_{\max} = 100\cdot 1_{d\times r}, Z^0 = 0_{d\times r}$ and $\epsilon_k = \max(10^{-5},0.9^k)$, where $k\in\mathbb{N}$ is the iteration counter. We terminated MIALM if  $\|X^k - Y^k\|_F^2 \leq 10^{-9}$ or $k\ge 500$. The $qr$ retraction is used in inner iteration of the  MIALM,   and a Barzilai-Borwein stepsize is used to accelerate it. The inner iteration is terminated if  $\|\mbox{grad}\Psi_{\bar{Z}^k}(X)\|_X \leq \epsilon_k$ or   the  iteration number exceeds $20$.  The final objective value obtained  by the MIALM method is denoted by $F_M$. 

For the MADMM, the penalty parameter is set to $\rho = \lambda_{\max}(H)/2$. We terminated MADMM if  $\|X^k - Y^k\|_F^2 \leq 10^{-9}$ or  $F(X^k)\leq F_M + 10^{-7}$, or $k\ge 500$. The inner iteration of the MADMM terminates if  the norm of Riemannian gradient of $X$-subproblem is less than $10^{-5}$ or  the inner iteration number exceeds $20$. For  the SOC, PAMAL and ManPG, the parameters are set   as same as in   \cite{chen2018proximal},  except   that the penalty parameter   $\rho = 2\lambda_{\max}(H)$ in SOC and PAMAL.  The ManPG terminates if stopping criterion described in \cite{chen2018proximal} is met or $F(X^k)\leq F_M + 10^{-7}$.  For easy comparisons, Table \ref{my-label4} lists the objective function value, sparsity of solution and cpu time. One can find from Table  \ref{my-label4}  that,  our MILAM method outperforms   to the other methods.
 
\begin{table}[tbp]
\centering
\begin{center}
\caption{ Comparisons of MIALM and ManPG, MADMM, PAMAL, SOC on CMs problem }
\label{my-label4}
\resizebox{1\textwidth}{90pt}{
\begin{tabular}{|l|l|l|l|l|l|l|l|l|l|l|}\toprule
\hline
&\multirow{2}{*}{$\mu$} & \multicolumn{3}{l|}{MIALM} & \multicolumn{3}{l|}{ManPG}       & \multicolumn{3}{l|}{MADMM} \\ \cline{3-11}
&                   & time    & $F_M$       & sp     & time     & F        & sp       & time    & F       & sp     \\ \hline
 \multirow{9}{*}{\rotatebox{270}{$r = 2,n=128$}}
&0.1                  & 0.021 	&	0.943 	&	0.835  & 0.036 	&	0.943 	&	0.836    & 0.112 	&	0.943 	&	0.836  \\
&0.2                  & 0.016 	&	1.639 	&	0.881  & 0.024 	&	1.639 	&	0.882    & 0.024 	&	1.639 	&	0.882  \\
&0.3                  & 0.020 	&	2.265 	&	0.901  & 0.029 	&	2.265 	&	0.900    & 0.167 	&	2.265 	&	0.903  \\ \cline{2-11}
&\multirow{2}{*}{$\mu$} & \multicolumn{3}{l|}{PAMAL} & \multicolumn{3}{l|}{SOC}       & \multicolumn{3}{l|}{} \\ \cline{3-8}
&                   & time    & F       & sp     & time     & F        & sp       & \multicolumn{3}{l|}{}     \\ \cline{2-8}
&0.1                  & 0.049 	&	0.943 	&	0.837  & 0.024 	&	0.943 	&	0.837    & \multicolumn{3}{l|}{}  \\
&0.2                  & 0.038 	&	1.639 	&	0.882  & 0.017 	&	1.639 	&	0.882    & \multicolumn{3}{l|}{}  \\
&0.3                  & 0.088 	&	2.265 	&	0.901  & 0.026 	&	2.265 	&	0.901   & \multicolumn{3}{l|}{}  \\ \hline
\end{tabular} }
\resizebox{1\textwidth}{90pt}{
\begin{tabular}{|l|l|l|l|l|l|l|l|l|l|l|}\toprule
\hline
&\multirow{2}{*}{$r$} & \multicolumn{3}{l|}{MIALM} & \multicolumn{3}{l|}{ManPG}       & \multicolumn{3}{l|}{MADMM} \\ \cline{3-11}
&                   & time    & $F_M$        & sp     & time     & F        & sp       & time    & F       & sp     \\ \hline
 \multirow{9}{*}{\rotatebox{270}{$\mu = 0.2,n=256$}}
&2                  & 0.021 	&	2.167 	&	0.892  & 0.071 	&	2.167 	&	0.892    & 0.153 	&	2.167 	&	0.892  \\
&4                  & 0.063 	&	4.334 	&	0.887  & 0.233 	&	4.334 	&	0.886    & 0.311 	&	4.338 	&	0.884  \\
&6                  & 0.345 	&	6.500 	&	0.889  & 0.722 	&	6.500 	&	0.884    & 0.531 	&	6.509 	&	0.881  \\ \cline{2-11}
&\multirow{2}{*}{r} & \multicolumn{3}{l|}{PAMAL} & \multicolumn{3}{l|}{SOC}       & \multicolumn{3}{l|}{} \\ \cline{3-8}
&                   & time    & F       & sp     & time     & F        & sp       & \multicolumn{3}{l|}{}        \\ \cline{2-8}
&2                  & 0.127 	&	2.167 	&	0.892  & 0.057 	&	2.167 	&	0.892    & \multicolumn{3}{l|}{}  \\
&4                  & 0.709 	&	4.334 	&	0.888  & 0.273 	&	4.334 	&	0.888    & \multicolumn{3}{l|}{}  \\
&6                  & 3.036 	&	6.500 	&	0.887  & 0.980 	&	6.500 	&	0.887    & \multicolumn{3}{l|}{}  \\ \hline
\end{tabular} }
\resizebox{1\textwidth}{90pt}{
\begin{tabular}{|l|l|l|l|l|l|l|l|l|l|l|}\toprule
\hline
&\multirow{2}{*}{$n$} & \multicolumn{3}{l|}{MIALM} & \multicolumn{3}{l|}{ManPG}       & \multicolumn{3}{l|}{MADMM} \\ \cline{3-11}
 &                   & time    & $F_M$        & sp     & time     & F        & sp       & time    & F       & sp     \\ \hline
 \multirow{9}{*}{\rotatebox{270}{$\mu = 0.6,r=2$}}
&200                  & 0.018 	&	2.265 	&	0.901  & 0.028 	&	2.265 	&	0.901    & 0.167 	&	2.265 	&	0.903  \\
&300                  & 0.017 	&	2.996 	&	0.910  & 0.051 	&	2.996 	&	0.910    & 0.128 	&	3.005 	&	0.909  \\
&500                  & 0.026 	&	3.956 	&	0.920  & 0.132 	&	3.956 	&	0.920    & 0.282 	&	4.048 	&	0.916  \\ \cline{2-11}
&\multirow{2}{*}{$n$} & \multicolumn{3}{l|}{PAMAL} & \multicolumn{3}{l|}{SOC}       & \multicolumn{3}{l|}{} \\ \cline{3-8}
&                   & time    & F       & sp     & time     & F        & sp       & \multicolumn{3}{l|}{}     \\ \cline{2-8}
&200                  & 0.045 	&	2.265 	&	0.902  & 0.028 	&	2.265 	&	0.901    & \multicolumn{3}{l|}{}  \\
&300                  & 0.085 	&	2.996 	&	0.910  & 0.041 	&	2.996 	&	0.910    & \multicolumn{3}{l|}{}  \\
&500                  & 0.253 	&	3.956 	&	0.920  & 0.137 	&	3.956 	&	0.920   & \multicolumn{3}{l|}{}  \\ \hline
\end{tabular} }
\end{center}
\end{table}


\subsection{Sparse principle component analysis}
Given a data set $\{b_1,\cdots, b_m\}$ where $b_i\in\mathbb{R}^{n\times 1}$. The sparse PCA problem is
\begin{equation}\label{spca}\left\{
\begin{split}
  \min_{X\in\mathbb{R}^{n\times r}} &  \sum_{i=1}^{m}\|b_i - XX^Tb_i\|_2^2 + \mu \|X\|_1, \\
  \mbox{s.t. }~~ & X^TX = I_r,
  \end{split}\right.
\end{equation}
where $\mu$ is a regularization parameter. Let $B = [b_1,\cdots, b_m ]^{T}\in\mathbb{R}^{m\times n}$, problem \eqref{spca} has the form: 
\begin{equation}\left\{
  \begin{split}
     \min_{X\in\mathbb{R}^{n\times r}} &  -tr(X^TB^TBX) + \mu \|X\|_1, \\
     \mbox{s.t. }~~ & X^TX = I_r.
  \end{split}\right.
\end{equation}

In our experiments, the random data matrix $B\in\mathbb{R}^{m\times n}$ is  generated by MATLAB function $randn(m, n)$, in which  the entries of $B$ follow the standard Gaussian distribution. We  shift the columns of $B$ such that they have mean $0$, and finally the column-vectors are normalized. The parameters of our MIALM are set as same as that of  used for the CMs problem,  except that  the stopping criterion is modified to   $\|X^k - Y^k\|_F^2 \leq 10^{-8}$ and the penalty parameter   $\rho_0 = \lambda_{\max}^2(B^TB)/2$.  Similarly,  the parameters of the MADMM are also set  as the same as that used for the CMs problem,  except  that  the penalty parameter  $\rho_0 = \lambda_{max}^2(B^TB)/2$.
For the SOC, PAMAL and ManPG methods,   the stopping criterion and parameter settings   provided in \cite{chen2018proximal} are copied. The interesting  readers are referred to \cite{chen2018proximal} for details.  Table \ref{my-label3} lists performance of  all  methods on the sparse PCA problem for comparisons.

\begin{table}[tbp]
\centering
\begin{center}
\caption{ Comparisons of MIALM and ManPG, MADMM, PAMAL, SOC on  SPCA   ($m = 50$) }
\label{my-label3}
\resizebox{1\textwidth}{90pt}{
\begin{tabular}{|l|l|l|l|l|l|l|l|l|l|l|}\toprule
\hline
&\multirow{2}{*}{$\mu$} & \multicolumn{3}{l|}{MIALM} & \multicolumn{3}{l|}{ManPG}       & \multicolumn{3}{l|}{MADMM} \\ \cline{3-11}
&                   & time    & $F_M$        & sp     & time     & F        & sp       & time    & F       & sp     \\ \hline
 \multirow{9}{*}{\rotatebox{270}{$r = 2,n=200$}}
&0.5                  & 0.038 	&	-6.839 	&	0.461  & 0.035 	&	-6.819 	&	0.458    & 0.193 	&	-6.767 	&	0.454  \\
&0.6                  & 0.038 	&	-5.304 	&	0.543  & 0.042 	&	-5.248 	&	0.545    & 0.201 	&	-5.147 	&	0.539  \\
&0.8                  & 0.043 	&	-2.439 	&	0.722  & 0.047 	&	-2.369 	&	0.732    & 0.199 	&	-2.285 	&	0.732  \\ \cline{2-11}
&\multirow{2}{*}{$\mu$} & \multicolumn{3}{l|}{PAMAL} & \multicolumn{3}{l|}{SOC}       & \multicolumn{3}{l|}{} \\ \cline{3-8}
&                   & time    & F       & sp     & time     & F        & sp       & \multicolumn{3}{l|}{}     \\ \cline{2-8}
&0.5                  & 1.919 	&	-6.847 	&	0.460  & 0.251 	&	-6.826 	&	0.458    & \multicolumn{3}{l|}{}  \\
&0.6                  & 2.123 	&	-5.267 	&	0.545  & 0.302 	&	-5.262 	&	0.544    & \multicolumn{3}{l|}{}  \\
&0.8                  & 2.247 	&	-2.387 	&	0.733  & 0.281 	&	-2.371 	&	0.732   & \multicolumn{3}{l|}{}  \\ \hline
\end{tabular} }
\resizebox{1\textwidth}{90pt}{
\begin{tabular}{|l|l|l|l|l|l|l|l|l|l|l|}\toprule
\hline
&\multirow{2}{*}{$r$} & \multicolumn{3}{l|}{MIALM} & \multicolumn{3}{l|}{ManPG}       & \multicolumn{3}{l|}{MADMM} \\ \cline{3-11}
&                   & time    & $F_M$        & sp     & time     & F        & sp       & time    & F       & sp     \\ \hline
 \multirow{9}{*}{\rotatebox{270}{$\mu = 0.6,n=200$}}
&2                  & 0.040 	&	-5.308 	&	0.548  & 0.039 	&	-5.290 	&	0.547    & 0.199 	&	-5.209 	&	0.538  \\
&3                  & 0.047 	&	-7.563 	&	0.562  & 0.058 	&	-7.530 	&	0.561    & 0.223 	&	-7.369 	&	0.552  \\
&5                  & 0.091 	&	-11.625 	&	0.594  & 0.117 	&	-11.571 	&	0.591    & 0.291 	&	-11.304 	&	0.582  \\ \cline{2-11}
&\multirow{2}{*}{r} & \multicolumn{3}{l|}{PAMAL} & \multicolumn{3}{l|}{SOC}       & \multicolumn{3}{l|}{} \\ \cline{3-8}
&                   & time    & F       & sp     & time     & F        & sp       & \multicolumn{3}{l|}{}        \\ \cline{2-8}
&2                  & 0.040 	&	-5.308 	&	0.548  & 0.251 	&	-5.329 	&	0.544    & \multicolumn{3}{l|}{}  \\
&3                  & 3.322 	&	-7.597 	&	0.562  & 0.442 	&	-7.552 	&	0.561    & \multicolumn{3}{l|}{}  \\
&5                  & 6.828 	&	-11.687 	&	0.592  & 0.674 	&	-11.727 	&	0.588    & \multicolumn{3}{l|}{}  \\ \hline
\end{tabular} }
\resizebox{1\textwidth}{90pt}{
\begin{tabular}{|l|l|l|l|l|l|l|l|l|l|l|}\toprule
\hline
&\multirow{2}{*}{$n$} & \multicolumn{3}{l|}{MIALM} & \multicolumn{3}{l|}{ManPG}       & \multicolumn{3}{l|}{MADMM} \\ \cline{3-11}
 &                   & time    & $F_M$        & sp     & time     & F        & sp       & time    & F       & sp     \\ \hline
 \multirow{9}{*}{\rotatebox{270}{$\mu = 0.6,r=2$}}
&200                  & 0.039 	&	-5.323 	&	0.539  & 0.040 	&	-5.283 	&	0.541    & 0.203 	&	-5.166 	&	0.538  \\
&300                  & 0.048 	&	-8.128 	&	0.473  & 0.043 	&	-8.112 	&	0.473    & 0.227 	&	-7.955 	&	0.467  \\
&500                  & 0.085 	&	-14.139 	&	0.399  & 0.054 	&	-14.134 	&	0.399    & 0.303 	&	-13.698 	&	0.385  \\ \cline{2-11}
&\multirow{2}{*}{$n$} & \multicolumn{3}{l|}{PAMAL} & \multicolumn{3}{l|}{SOC}       & \multicolumn{3}{l|}{} \\ \cline{3-8}
&                   & time    & F       & sp     & time     & F        & sp       & \multicolumn{3}{l|}{}     \\ \cline{2-8}
&200                  & 2.187 	&	-5.282 	&	0.545  & 0.288 	&	-5.290 	&	0.542    & \multicolumn{3}{l|}{}  \\
&300                  & 3.037 	&	-8.106 	&	0.477  & 0.443 	&	-8.108 	&	0.474    & \multicolumn{3}{l|}{}  \\
&500                  & 9.618 	&	-14.106 	&	0.400  & 1.283 	&	-14.109 	&	0.398    & \multicolumn{3}{l|}{}  \\ \hline
\end{tabular} }
\end{center}
\end{table}

\section{Conclusions}\label{sec-con}
We proposed a manifold inexact augmented Lagrangian method for nonsmooth composite minimization problem on Riemannian manifold. At each iteration of the proposed method,  we only need to solve a smooth Riemannian manifold minimization subproblem based on  the Moreau envelope.  The convergence of the proposed method is established under some mild assumptions. Numerical experiments show that,  the proposed method is competitive  compared to some existing state-of-the-art methods. 

{\small
\bibliographystyle{abbrv}
\bibliography{ref}

\begin{thebibliography}{10}

\bibitem{Absil2007Trust}
P.-A. Absil, C.~G. Baker, and K.~A. Gallivan.
\newblock Trust-region methods on riemannian manifolds.
\newblock {\em Foundations of Computational Mathematics}, 7(3):303--330, 2007.

\bibitem{absil2017collection}
P.-A. Absil and S.~Hosseini.
\newblock A collection of nonsmooth riemannian optimization problems.
\newblock In {\em Nonsmooth Optimization and Its Applications}, pages 1--15.
  Springer, 2019.

\bibitem{AbsMahSep2008}
P.-A. Absil, R.~Mahony, and R.~Sepulchre.
\newblock {\em Optimization algorithms on matrix manifolds}.
\newblock Princeton University Press, 2009.

\bibitem{beck2017first}
A.~Beck.
\newblock {\em First-order methods in optimization}, volume~25.
\newblock SIAM, 2017.

\bibitem{Bento2017Iteration}
G.~C. Bento, O.~P. Ferreira, and J.~G. Melo.
\newblock Iteration-complexity of gradient, subgradient and proximal point
  methods on riemannian manifolds.
\newblock {\em Journal of Optimization Theory and Applications},
  173(2):548--562, 2017.

\bibitem{bento2011convergence}
G.~d.~C. Bento, J.~C. d.~L. Neto, and P.~R. Oliveira.
\newblock Convergence of inexact descent methods for nonconvex optimization on
  riemannian manifolds.
\newblock {\em arXiv preprint arXiv:1103.4828}, 2011.

\bibitem{bertsekas1997nonlinear}
D.~P. Bertsekas.
\newblock Nonlinear programming.
\newblock {\em Journal of the Operational Research Society}, 48(3):334--334,
  1997.

\bibitem{Bortoloti2018Damped}
M.~Bortoloti, T.~A. Fernandes, O.~P. Ferreira, and J.~Yuan.
\newblock Damped newton's method on riemannian manifolds.
\newblock {\em arXiv preprint arXiv:1803.05126}, 2018.

\bibitem{Boumal2015Riemannian}
N.~Boumal.
\newblock Riemannian trust regions with finite-difference hessian
  approximations are globally convergent.
\newblock In {\em International Conference on Geometric Science of
  Information}, pages 467--475, 2015.

\bibitem{Boumal2016Global}
N.~Boumal, P.-A. Absil, and C.~Cartis.
\newblock Global rates of convergence for nonconvex optimization on manifolds.
\newblock {\em IMA Journal of Numerical Analysis}, 39(1):1--33, 2018.

\bibitem{JMLR:v15:boumal14a}
N.~Boumal, B.~Mishra, P.-A. Absil, and R.~Sepulchre.
\newblock Manopt, a matlab toolbox for optimization on manifolds.
\newblock {\em The Journal of Machine Learning Research}, 15(1):1455--1459,
  2014.

\bibitem{burke2018gradient}
J.~V. Burke, F.~E. Curtis, A.~S. Lewis, M.~L. Overton, and L.~E. Sim{\~o}es.
\newblock Gradient sampling methods for nonsmooth optimization.
\newblock {\em arXiv preprint arXiv:1804.11003}, 2018.

\bibitem{cambier2016robust}
L.~Cambier and P.-A. Absil.
\newblock Robust low-rank matrix completion by riemannian optimization.
\newblock {\em SIAM Journal on Scientific Computing}, 38(5):S440--S460, 2016.

\bibitem{chen2018proximal}
S.~Chen, S.~Ma, A.~M.-C. So, and T.~Zhang.
\newblock Proximal gradient method for manifold optimization.
\newblock {\em arXiv preprint arXiv:1811.00980}, 2018.

\bibitem{chen2016augmented}
W.~Chen, H.~Ji, and Y.~You.
\newblock An augmented lagrangian method for 1-regularized optimization
  problems with orthogonality constraints.
\newblock {\em SIAM Journal on Scientific Computing}, 38(4):B570--B592, 2016.

\bibitem{deCarvalhoBento2016}
G.~de~Carvalho~Bento, J.~X. da~Cruz~Neto, and P.~R. Oliveira.
\newblock A new approach to the proximal point method: convergence on general
  riemannian manifolds.
\newblock {\em Journal of Optimization Theory and Applications},
  168(3):743--755, 2016.

\bibitem{dhingra2018proximal}
N.~K. Dhingra, S.~Z. Khong, and M.~R. Jovanovic.
\newblock The proximal augmented lagrangian method for nonsmooth composite
  optimization.
\newblock {\em IEEE Transactions on Automatic Control}, 2018.

\bibitem{ferreira2019iteration}
O.~P. Ferreira, M.~S. Louzeiro, and L.~Prudente.
\newblock Iteration-complexity of the subgradient method on riemannian
  manifolds with lower bounded curvature.
\newblock {\em Optimization}, 68(4):713--729, 2019.

\bibitem{gohary2009noncoherent}
R.~H. Gohary and T.~N. Davidson.
\newblock Noncoherent mimo communication: Grassmannian constellations and
  efficient detection.
\newblock {\em IEEE Transactions on Information Theory}, 55(3):1176--1205,
  2009.

\bibitem{Grohs2016Nonsmooth}
P.~Grohs and S.~Hosseini.
\newblock Nonsmooth trust region algorithms for locally lipschitz functions on
  riemannian manifolds.
\newblock {\em IMA Journal of Numerical Analysis}, 36(3):1167--1192, 2015.

\bibitem{Grohs2015}
P.~Grohs and S.~Hosseini.
\newblock $\varepsilon$-subgradient algorithms for locally lipschitz functions
  on riemannian manifolds.
\newblock {\em Advances in Computational Mathematics}, 42(2):333--360, 2016.

\bibitem{Hong2017Nonconvex}
Z.~Hong, X.~Zhang, D.~Chu, and L.~Liao.
\newblock Nonconvex and nonsmooth optimization with generalized orthogonality
  constraints: An approximate augmented lagrangian method.
\newblock {\em Journal of Scientific Computing}, 72(1):1--42, 2017.

\bibitem{hosseini2015convergence}
S.~Hosseini.
\newblock Convergence of nonsmooth descent methods via kurdyka--lojasiewicz
  inequality on riemannian manifolds.
\newblock {\em Hausdorff Center for Mathematics and Institute for Numerical
  Simulation, University of Bonn (2015,(INS Preprint No. 1523))}, 2015.

\bibitem{doi:10.1137/16M1108145}
S.~Hosseini, W.~Huang, and R.~Yousefpour.
\newblock Line search algorithms for locally lipschitz functions on riemannian
  manifolds.
\newblock {\em SIAM Journal on Optimization}, 28(1):596--619, 2018.

\bibitem{doi:10.1137/16M1069298}
S.~Hosseini and A.~Uschmajew.
\newblock A riemannian gradient sampling algorithm for nonsmooth optimization
  on manifolds.
\newblock {\em SIAM Journal on Optimization}, 27(1):173--189, 2017.

\bibitem{huang2013optimization}
W.~Huang.
\newblock Optimization algorithms on riemannian manifolds with applications.
\newblock 2013.

\bibitem{Huang2015A}
W.~Huang, P.-A. Absil, and K.~A. Gallivan.
\newblock A riemannian symmetric rank-one trust-region method.
\newblock {\em Mathematical Programming}, 150(2):179--216, 2015.

\bibitem{HuaGalAbs2015}
W.~Huang, K.~A. Gallivan, and P.-A. Absil.
\newblock A broyden class of quasi-newton methods for riemannian optimization.
\newblock {\em SIAM Journal on Optimization}, 25(3):1660--1685, 2015.

\bibitem{Kovna2015}
A.~Kovnatsky, K.~Glashoff, and M.~M. Bronstein.
\newblock Madmm: a generic algorithm for non-smooth optimization on manifolds.
\newblock In {\em European Conference on Computer Vision}, pages 680--696.
  Springer, 2016.

\bibitem{Lai2014A}
R.~Lai and S.~Osher.
\newblock A splitting method for orthogonality constrained problems.
\newblock {\em Journal of Scientific Computing}, 58(2):431--449, 2014.

\bibitem{li2018highly}
X.~Li, D.~Sun, and K.-C. Toh.
\newblock A highly efficient semismooth newton augmented lagrangian method for
  solving lasso problems.
\newblock {\em SIAM Journal on Optimization}, 28(1):433--458, 2018.

\bibitem{lin2019efficient}
M.~Lin, Y.-J. Liu, D.~Sun, and K.-C. Toh.
\newblock Efficient sparse semismooth newton methods for the clustered lasso
  problem.
\newblock {\em SIAM Journal on Optimization}, 29(3):2026--2052, 2019.

\bibitem{ozolicnvs2013compressed}
V.~Ozoli{\c{n}}{\v{s}}, R.~Lai, R.~Caflisch, and S.~Osher.
\newblock Compressed modes for variational problems in mathematics and physics.
\newblock {\em Proceedings of the National Academy of Sciences},
  110(46):18368--18373, 2013.

\bibitem{witten2009penalized}
D.~M. Witten, R.~Tibshirani, and T.~Hastie.
\newblock A penalized matrix decomposition, with applications to sparse
  principal components and canonical correlation analysis.
\newblock {\em Biostatistics}, 10(3):515--534, 2009.

\bibitem{Yang2014}
W.~H. Yang, L.-H. Zhang, and R.~Song.
\newblock Optimality conditions for the nonlinear programming problems on
  riemannian manifolds.
\newblock {\em Pacific Journal of Optimization}, 10(2):415--434, 2014.

\bibitem{YHAG2017}
X.~Yuan, W.~Huang, P.-A. Absil, and K.~Gallivan.
\newblock A riemannian quasi-newton method for computing the karcher mean of
  symmetric positive definite matrices.
\newblock {\em Florida State University,(FSU17-02)}, 2017.

\bibitem{Zhang2016First}
H.~Zhang and S.~Sra.
\newblock First-order methods for geodesically convex optimization.
\newblock In {\em Conference on Learning Theory}, pages 1617--1638, 2016.

\bibitem{Zhang2019}
J.~Zhang, S.~Ma, and S.~Zhang.
\newblock Primal-dual optimization algorithms over riemannian manifolds: an
  iteration complexity analysis.
\newblock {\em arXiv preprint arXiv:1710.02236}, 2017.

\bibitem{zheng2002communication}
L.~Zheng and D.~N.~C. Tse.
\newblock Communication on the grassmann manifold: A geometric approach to the
  noncoherent multiple-antenna channel.
\newblock {\em IEEE Transactions on Information Theory}, 48(2):359--383, 2002.

\bibitem{zou2006sparse}
H.~Zou, T.~Hastie, and R.~Tibshirani.
\newblock Sparse principal component analysis.
\newblock {\em Journal of computational and graphical statistics},
  15(2):265--286, 2006.

\end{thebibliography}
}

\end{document}